\documentclass[12pt]{article}
\usepackage[utf8]{inputenc}	
\usepackage{amsmath,amsthm,amsfonts,amssymb,amscd}
\usepackage{multirow,booktabs}
\usepackage[table]{xcolor}
\usepackage{fullpage}
\usepackage{lastpage}
\usepackage{enumitem}
\usepackage{fancyhdr}
\usepackage{mathrsfs}
\usepackage{wrapfig}
\usepackage{setspace}
\usepackage{subcaption}
\usepackage{graphicx}
\usepackage{calc}
\usepackage{multicol}
\usepackage{cancel}

\usepackage[most]{tcolorbox}
\usepackage{xcolor}
\usepackage{hyperref}
\usepackage{tikz}
\usepackage{tikz-cd}
\usepackage{caption}
\usepackage{accents}

\pdfoutput=1

\newtheorem{theorem}{Theorem}[section]

\newtheorem{corollary}[theorem]{Corollary}
\newtheorem{definition}[theorem]{Definition}
\newtheorem{lemma}[theorem]{Lemma}
\newtheorem{proposition}[theorem]{Proposition}
{Important Convention}
\theoremstyle{remark}
\newtheorem{remark}[theorem]{Remark}

\DeclareMathOperator{\supp}{supp}

\DeclareMathOperator{\id}{Id}

\newcommand{\Gcal}{\mathcal{G}}
\newcommand{\Hcal}{\mathcal{H}}

\newcommand{\Pcal}{\mathcal{P}}

\newcommand{\Vcal}{{V}}
\newcommand{\Wcal}{\mathcal{W}}

\newcommand{\E}{\mathbb{E}}

\newcommand{\N}{\mathbb{N}}

\newcommand{\PP}{\mathbb{P}}

\newcommand{\R}{\mathbb{R}}

\usepackage{bbm}

\renewcommand{\epsilon}{\varepsilon}

\newcommand{\Cpl}{\text{Cpl}}
\newcommand{\Law}{\mathscr L}

\newcommand{\MCov}{{\rm MCov}}
\newcommand{\bary}{{\rm bary}}

\newcommand{\Functional}{\mathcal V}

\DeclareMathOperator*{\esssup}{ess\,sup}
\DeclareMathOperator*{\essinf}{ess\,inf}

\title{The Gradient Flow of the Bass Functional in Martingale Optimal Transport}

\author{Julio Backhoff-Veraguas, Gudmund Pammer, Walter Schachermayer}

\begin{document}

\maketitle

\begin{abstract}

    {
    Given $\mu$ and $\nu$, probability measures on $\R^d$ in convex order, a Bass martingale is arguably the most natural martingale starting with law $\mu$ and finishing with law $\nu$. Indeed, this martingale is obtained by \emph{stretching} a reference Brownian motion so as to meet the data $\mu,\nu$. Unless $\mu$ is a Dirac, the existence of a Bass martingale is a delicate subject, since for instance the reference Brownian motion must be allowed to have a non-trivial initial distribution $\alpha$, not known in advance. Thus the key to obtaining the Bass martingale, theoretically as well as practically, lies in finding $\alpha$.  

    In \cite{BaSchTsch23} it has been shown that $\alpha$ is determined as the minimizer of the so-called \emph{Bass functional}. In the present paper we propose to minimize this functional by following its gradient flow, or more precisely, the gradient flow of its $L^2$-lift. In our main result we show that this gradient flow converges in norm to a minimizer of the Bass functional, and when $d= 1$ we further establish that convergence is exponentially fast. 
    }
\end{abstract}

\section{Introduction}

\subsection{Martingale optimization problem}

Let $\mu, \nu$ be elements of $\mathcal P_{2}(\R^d)$, the space of probability measures on $\R^d$ with finite second moments. Assume that $\mu, \nu$ are in convex order, denoted by $\mu \le_{cx} \nu$, and meaning that $\int \phi \, d\mu \leqslant \int \phi \, d\nu$ holds for all convex functions $\phi \colon \R^d \rightarrow \R$. As in \cite{BaBeHuKa20,BBST23} we consider the martingale optimization problem 
\begin{equation*} \label{MBMBB} \tag{MBB}
MT(\mu, \nu) := 
\inf_{\substack{M_{0} \sim \mu, \, M_{1} \sim \nu, \\ M_{t} = M_{0} + \int_{0}^{t}  \sigma_{s} \, dB_{s}}} 
\mathbb{E}\Big[\int_{0}^{1} \vert \sigma_{t} - I_{d} \vert^{2}_{\textnormal{HS}} \, dt\Big],
\end{equation*} 
where $B$ is a $d$-dimensional Brownian motion and $\vert \cdot \vert_{\textnormal{HS}}$ denotes the Hilbert--Schmidt norm. The abbreviation ``MBB'' stands for ``Martingale Benamou--Brenier'' and this designation is motivated from the fact that \eqref{MBMBB} can be seen as a martingale counterpart of the classical formulation in optimal transport by Benamou--Brenier \cite{BeBr99}, see \cite{BaBeHuKa20,BBST23}. The problem \eqref{MBMBB} can be shown to be equivalent to 
\begin{equation} \label{MBMBB2}
P(\mu, \nu) := 
\sup_{\substack{M_0 \sim \mu, \, M_1 \sim \nu, \\ M_t = M_0 + \int_0^t \sigma_s \, dB_s}}
\mathbb{E}\Big[\int_0^1 \textnormal{tr} (\sigma_t) \, dt\Big].
\end{equation} 
Indeed their optimizer are identical and we have $MT(\mu,\nu)=-2P(\mu, \nu) +\int|x|^2d(\nu-\mu) +d $.

As shown in \cite{BaBeHuKa20}, the problem \eqref{MBMBB} admits a strong Markov martingale $\hat{M}$ as the unique optimizer, which is called the \textit{stretched Brownian motion} from $\mu$ to $\nu$ in \cite{BaBeHuKa20}.

\subsection{Bass martingales and structure of stretched Brownian motion}

Owing to the work \cite{BBST23} it is known that the optimality property of stretched Brownian motion is related to a structural / geometric description. For its formulation we start with the following definition.

\begin{definition} \label{defi:irreducible_intro} 
For probability measures $\mu, \nu$ we say that the pair $(\mu,\nu)$ is \textit{irreducible} if for all measurable sets $A, B \subseteq \R^d$ with $\mu(A), \nu(B)>0$ there is a martingale $X= (X_{t})_{0 \leqslant t \leqslant 1}$ with $X_{0} \sim \mu$, $X_{1} \sim \nu$ such that $\mathbb{P}(X_{0}\in A, X_{1}\in B) >0$. 
\end{definition} 

We remark that in the classical theory of optimal transport one can always find couplings $(X_{0},X_{1})$ of $(\mu,\nu)$ such that $\mathbb{P}(X_{0}\in A, X_{1}\in B) > 0$, for all measurable sets $A, B \subseteq \R^d$ with $\mu(A),\nu(B) > 0$; e.g., by letting $(X_{0},X_{1})$ be independent. In martingale optimal transport this property may fail.

\smallskip

Next we recall the following concept from \cite{Ba83, BaBeHuKa20, BBST23}.

\begin{definition} \label{def:BassMarti_intro} Let $B = (B_{t})_{0 \leqslant t \leqslant 1}$ be a Brownian motion on $\R^d$ with $B_{0} \sim \hat{\alpha}$, where $\hat{\alpha}$ is an arbitrary element of $\mathcal P(\R^d)$, the space of probability measures on $\R^d$. Let $\hat{v} \colon \R^d \rightarrow \R$ be convex such that $\nabla \hat{v}(B_{1})$ is square-integrable. We call 
\begin{equation} \label{def:BassMarti_intro_eq} 
\hat{M}_{t} := 
\E[\nabla \hat{v}(B_{1}) \, \vert \, \sigma(B_{s} \colon s \leqslant t)]
= \E[\nabla \hat{v}(B_{1}) \, \vert \, B_{t}], \qquad 0 \leqslant t \leqslant 1
\end{equation}
a \textit{Bass martingale} with \textit{Bass measure} $\hat{\alpha}$ joining $\mu = \Law(\hat{M}_{0})$ with $\nu = \Law(\hat{M}_{1})$.
\end{definition}

The reason behind this terminology is that Bass \cite{Ba83} used this construction (with $d=1$ and $\hat{\alpha}$ a Dirac measure) in order to derive a solution of the Skorokhod embedding problem.

\smallskip

In \cite[Theorem 1.3]{BBST23} it is shown that under the irreducibility assumption on the pair $(\mu,\nu)$ there is a unique Bass martingale $\hat{M}$ from $\mu$ to $\nu$, i.e., satisfying $\hat{M}_{0} \sim \mu$ and $\hat{M}_{1} \sim \nu$:

\begin{theorem} \label{MainTheorem} Let $\mu, \nu \in \mathcal P_{2}(\R^d)$ with $\mu \le_{cx} \nu$ and assume that $(\mu,\nu)$ is irreducible. Then the following are equivalent for a martingale $\hat{M}=(\hat{M}_{t})_{0 \leqslant t \leqslant 1}$ with $\hat{M}_{0} \sim \mu$ and $\hat{M}_{1} \sim \nu$:
\begin{enumerate}[label=(\arabic*)] 
\item \label{MainTheorem_1} $\hat{M}$ is a stretched Brownian motion, i.e., the optimizer of \eqref{MBMBB}.
\item \label{MainTheorem_2} $\hat{M}$ is a Bass martingale.  
\end{enumerate}
\end{theorem}

Since, for probability measures $\mu, \nu \in \mathcal P_{2}(\R^d)$ with $\mu \le_{cx} \nu$, a stretched Brownian motion always exists by \cite[Theorem 1.5]{BaBeHuKa20}, the above theorem states that the existence of a Bass martingale follows from --- and is in fact equivalent to --- the irreducibility assumption on the pair $(\mu,\nu)$. 

\smallskip

Denoting by $\ast$ the convolution operator\footnote{By a minor abuse of notation we write $f \ast \gamma_1$ for the convolution of a measurable map $f \colon \R^d \to \R^d$ with the density of the Gaussian $\gamma_1$ and similarly $\alpha \ast \gamma_1$ for the convolution of a measure $\alpha \in \Pcal(\R^d)$ with the Gaussian measure $\gamma_1$.} and by $\gamma_1$ the standard Gaussian measure on $\R^d$, we remark that the convex function $\hat{v}$ and the Bass measure $\hat{\alpha}$ from Definition \ref{def:BassMarti_intro} satisfy the identities (see \cite{BBST23})
\begin{equation} 
\label{eq_def_id_bm}
(\nabla \hat{v} \ast \gamma_1)(\hat{\alpha}) = \mu
\qquad \textnormal{ and } \qquad 
\nabla \hat{v}(\hat{\alpha} \ast \gamma_1) = \nu.
\end{equation}
In terms of \eqref{def:BassMarti_intro_eq} this amounts to
\[
\hat{M}_{t} = \nabla \hat{v}\ast \gamma_{1-t}(B_{t}), \qquad 0 \leqslant t \leqslant 1,
\]
where $\gamma_s$ denotes the centred Gaussian with covariance matrix $sI_{d}$.

\subsection{The Bass functional} In the following we denote by $\MCov$ the \textit{maximal covariance}  between two probability measures $p_{1},p_{2} \in \mathcal P_{2}(\R^d)$, defined as
\begin{equation} \label{eq_def_mcov}
\MCov(p_{1},p_{2}) := \sup_{q \in \Cpl(p_{1},p_{2})} \int \langle x_{1},x_{2} \rangle \, q(dx_{1},dx_{2}),
\end{equation}
where $\Cpl(\mu,\nu)$ is the set of all couplings $\pi \in \mathcal P(\R^d \times \R^d)$ between $\mu$ and $\nu$, i.e., probability measures on $\R^d \times \R^d$ with first marginal $\mu$ and second marginal $\nu$. As is well known, maximizing the covariance between $p_{1}$ and $p_{2}$ is equivalent to minimizing their expected squared Wasserstein distance. We follow \cite{BaSchTsch23} in defining:

\begin{definition} The \textit{Bass functional} is given by
\begin{equation} \label{def.bass.func}
\mathcal P_{2}(\R^{d}) \ni \alpha \longmapsto \Functional(\alpha) 
:= \MCov(\alpha \ast \gamma_1, \nu) - \MCov(\alpha,\mu).
\end{equation}
\end{definition}
If $\mu \le_{c} \nu$ we have $\Functional(\alpha) \ge 0$ for every $\alpha \in \Pcal_2(\R^d)$.

The main result of \cite{BaSchTsch23} was the reformulation of Problem \eqref{MBMBB2}, which characterizes the Bass measure $\hat{\alpha}$ in \eqref{eq_def_id_bm} as the optimizer of the Bass functional \eqref{def.bass.func}. This is the content of \cite[Theorem 1.5]{BaSchTsch23}:

\begin{theorem} \label{prop_alpha_vari_mc} Let $\mu, \nu \in \mathcal P_{2}(\R^d)$ with $\mu \le_{cx} \nu$. Then
\begin{equation} \label{eq:variational_alpha_mc}
P(\mu,\nu) 
= \inf_{\alpha \in \mathcal P_{2}(\R^{d})} \Functional(\alpha).
\end{equation}
The right-hand side of \eqref{eq:variational_alpha_mc} is attained by $\hat{\alpha} \in \mathcal P_{2}(\R^{d})$ if and only if there is a Bass martingale from $\mu$ to $\nu$ with Bass measure $\hat{\alpha} \in \mathcal P_{2}(\R^{d})$.
\end{theorem}

\subsection{Main Results}

Per Theorem \ref{prop_alpha_vari_mc}, it is desirable to find the minimizer of the Bass functional. Indeed, doing so will provide us with the Bass measure and hence with the Bass martingale between $\mu$ and $\nu$. Moreover, one would like to have a \emph{fast} minimization method for the Bass functional. One idea would be to perform the gradient descent of the Bass functional in the sense of Wasserstein spaces. This seems challenging to the authors, since not much is known about the convexity of the Bass functional in the Wasserstein geometry for $d$ arbitrary (see however \cite[Theorem 1.7]{BaSchTsch23}). For this reason we take a detour. 

We fix a probability space supporting independent random variables $X\sim \mu$ and $\Gamma\sim \gamma_1$, which we fix throughout. We then define the sigma-algebras $\Gcal := \sigma(X)$ and $\Hcal:= \sigma(X,\Gamma)$, with their associated $L^2$-spaces of $\R^d$-valued random variables, respectively $L^2(\Hcal;\R^d)$ and $L^2(\Gcal;\R^d)$.
\begin{definition}
The lifted Bass functional is given by
   \begin{equation}
    \label{eq:def.Bass.Hilbert.functional.intro}
    L^2(\Gcal;\R^d)\ni Z \mapsto \Vcal(Z) := \sup_{ \substack{ Y \in L^2(\Hcal;\R^d), \\ Y \sim \nu } }\mathbb E[ (Z + \Gamma) \cdot Y - Z \cdot X].
\end{equation} 
\end{definition}
As we show in Lemma \ref{lem:Vcal} below, minimizing $\Functional$ or $\Vcal$ lead to the same problem, and in fact if $Z^*$ minimizes $\Vcal$ then $\Law(Z^*)$ is a Bass measure. Furthermore, we shall show that $\Vcal$ is convex in the conventional sense, Fr\'echet differentiable,  and admits the explicit $L^2$-gradient, $$D_Z \Vcal = \nabla v \ast \gamma_1(Z) - X,$$   where $\nabla v$ is the Brenier map\footnote{The Brenier map from an absolutely continuous measure $\rho$, to $\eta$, is the unique almost-sure gradient of a convex function $v$ satisfying the push-forward constraint $(\nabla v) (\rho) = \eta$, while $v$ is called the Brenier potential.} from $\Law(Z)\ast\gamma_1$ to $\nu$. We can hence define the $L^2$-Bass gradient flow as follows:

\begin{definition}
    An absolutely continuous curve $\{Z_t\}_{t\in\R_+}\subset L^2(\Gcal;\R^d)$ is called an $L^2$-Bass gradient flow (started at $Z_0$) if the curve satisfies for a.e.\ $t\in\R_+$ the equality
    \[ \frac{d Z_t}{dt}= - [\nabla v_t \ast \gamma_1(Z) - X], \]
    where $\nabla v_t$ is the Brenier map from $\Law(Z_t)\ast\gamma_1$ to $\nu$.
\end{definition}

Invoking the classical theory of gradient flows of convex functionals in Hilbert spaces, one could hope to obtain the existence and uniqueness of the Bass gradient flow, as well as the weak-$L^2$ convergence of the flow to a minimizer of $\Vcal$. However the story is not that simple, since there is in general no reason for $\Vcal$ to attain its minimum at a square integrable random variable. Moreover weak-$L^2$ convergence of the Bass flow is not strong enough to derive convergence of $\Law(Z_t)$ in a meaningful sense, and the abstract theory would not give   us convergence rates. For these reasons we need to make further assumptions and argue in a self-contained fashion.
\medskip

Under technical assumptions, defined in Section \ref{sec:gradient flow} below,\footnote{In short: the pair $(\mu,\nu)$ is irreducible, the convex hull of the support of $\nu$, denoted by $C$, is compact, and the convex hull of the support of $\mu$ is compactly contained in the interior of $C$.} we obtain our first main result:

\begin{theorem}
\label{thm:first_main}
 Under Assumption $(A)$, for any r.v.\ $Z_0\in L^2(\Gcal;\R^d)$, the $L^2$-Bass gradient flow $(Z_t)_{t \ge 0}$ exists, is unique, and it converges in $L^2$-norm to $Z^\star\in L^\infty(\Gcal;\R^d)$, where the latter is the unique minimizer of $\Vcal$ with $\text{bary}(Z^\star)=\text{bary}(Z_0)$.
 Further, if $Z_0$ is bounded, then the entire curve $(Z_t)_{t \ge 0}$ is uniformly bounded.
\end{theorem}

In our second main result, we sharpen the above by obtaining exponential convergence in the one-dimensional case:

\begin{theorem}
\label{thm:second_main} Let $d=1$. 
    In the setting of Theorem \ref{thm:first_main}, assume additionally that $Z_0\in L^\infty(\Gcal;\R^d)$, and that $\nu$ has convex support and admits a density which is bounded away from zero and infinity on its support. Then we have  for some $\kappa>0$:
    \begin{align}
    \| Z_t-Z^\star \|_{L^ 2}&\leq \| Z_0-Z^\star \|_{L^ 2}\cdot\exp\{-\kappa t\},\\
    \Vcal(Z_t)-\inf \Vcal &\leq [\Vcal(Z_0) - \inf \Vcal ]\cdot\exp\{-\kappa t\}.
    \end{align}
\end{theorem}
In order to obtain the latter result a careful study of the Hessian of $\Vcal$ must be carried out, which we deem of independent interest. A corollary of this analysis is the proof that $\Vcal$ is strongly convex when we restrict it to random variables which are uniformly bounded. 
We conjecture the validity of Theorem \ref{thm:second_main} also in higher dimensions. We leave the latter as well as the practical implementation of the Bass gradient flow as an interesting research topic, the exponential convergence result suggesting a fast computational method.

\subsection{Related literature} Optimal transport as a field in mathematics goes back to Monge \cite{Mo81} and Kantorovich \cite{Ka42}, who established its modern formulation. The seminal results of Benamou, Brenier, and McCann \cite{Br87, Br91, BeBr99, Mc94, Mc95} form the basis of the modern theory, with striking applications in a variety of different areas, see the monographs \cite{Vi03, Vi09, AmGi13, Sa15}. 

\smallskip

We are interested in transport problems where the transport plan satisfies an additional martingale constraint. This additional requirement arises naturally in finance (e.g.\ \cite{BeHePe12}), but is also of independent mathematical interest. For example there are notable consequences for the study of martingale inequalities (e.g.\ \cite{BoNu13,HeObSpTo12,ObSp14}) and the Skorokhod embedding problem (e.g.\ \cite{BeCoHu14, KaTaTo15, BeNuSt19}). Early articles on this topic of \textit{martingale optimal transport} include \cite{HoNe12, BeHePe12, TaTo13, GaHeTo13, DoSo12, CaLaMa14}. The study of irreducibility of a pair of marginals $(\mu,\nu)$ was initiated by Beiglb\"ock and Juillet \cite{BeJu16} in dimension one and extended in the works \cite{GhKiLi19,DeTo17,ObSi17} to multiple dimensions.

\smallskip

Continuous-time martingale optimal transport problems have received much attention in the recent years; see e.g.\ \cite{BeHeTo15, CoObTo19, GhKiPa19,GuLoWa19, GhKiLi20, ChKiPrSo20, GuLo21}. In this paper we concern ourselves with the specific structure given by the martingale Benamou--Brenier problem, introduced in \cite{BaBeHuKa20} in probabilistic language and in \cite{HuTr17} in PDE language, and subsequently studied through the point of view of duality theory in \cite{BBST23}. In the context of market impact in finance, the same kind of problem appeared independently in a work by Loeper \cite{Lo18}.
It was also shown in \cite{BaBeHuKa20} that the optimizer $\hat{M}$ of the problem \eqref{MBMBB} is the process whose evolution follows the movement of Brownian motion as closely as possible with respect to an \textit{adapted Wasserstein distance} (see e.g.\ \cite{BaBaBeEd19a, Fo22a}) subject to the given marginal constraints. 

\smallskip

The connection between the martingale Benamou--Brenier problem and Bass martingales was revealed in \cite{BaBeHuKa20,BBST23}. If $d=1$ then an algorithm for obtaining a Bass martingale with given marginals was proposed by Conze and Henry-Labordere in \cite{CoHL21}. This algorithm was subsequently studied in detail in \cite{AcMaPa23}, and its multidimensional form has been proposed in \cite{JoLoOb23}. Recently, objects akin to Bass martingales have been proposed in \cite{Tsch24}, \cite{BoGu24}, and \cite{BaLoOb24,BePaRi24}. 

\section{Regularity of the lifted Bass functional}
Throughout this paper, let $\mu, \nu \in \Pcal_2(\R^d)$.
We recall from \eqref{def.bass.func} the associated Bass functional $\Functional \colon \Pcal_2(\R^d) \to \R$ given by
\begin{equation}
    \label{eq:def.Bass.functional}
    \Functional(\alpha) := \MCov(\alpha \ast \gamma_1, \nu) - \MCov(\alpha,\mu).
\end{equation}
Assuming that $\mu$ and $\nu$ are in convex order ($\mu \le_{cx} \nu$), this functional takes its values in $[0,\infty)$.

In a first step, we will identify $\Functional$ with a convex, Frechet differentiable functional $\Vcal$ on a Hilbert space.
To this end, we fix a probability space $(\Omega,\mathcal F,\PP)$ with a random variable $X \sim \mu$ and an independent random variable $\Gamma \sim \gamma_1$ where $\gamma_1$ is here the $d$-dimensional standard Gaussian.
Define the sigma-algebras \[\Gcal := \sigma(X)\,\text{ and }\,\Hcal:= \sigma(X,\Gamma).\]
We define $\Vcal \colon L^2(\Gcal;\R^d) \to \R$ by

\begin{equation}
    \label{eq:def.Bass.Hilbert.functional}
    \Vcal(Z) := \sup_{ \substack{ Y \in L^2(\Hcal;\R^d), \\ Y \sim \nu } }\mathbb E[ (Z + \Gamma) \cdot Y - Z \cdot X]
\end{equation}

\begin{lemma}\label{lem:Vcal}
    Let $Z \in L^2(\Gcal;\R^d)$.
    The functional $\Vcal$ given in \eqref{eq:def.Bass.Hilbert.functional} enjoys the following properties:
    \begin{enumerate}[label = (\roman*)]
        \item \label{it:lem.Fprops.Frechet} $\Vcal$ is Frechet differentiable with derivative $D_Z \Vcal = \nabla v \ast \gamma_1(Z) - X \in L^2(\Gcal;\R^d)$ where $v$ is the Brenier potential from $\Law(Z + B)$ to $\nu$.
        \item \label{it:lem.Fprops.monotone}  If the law of $(X,Z)$ is monotone,\footnote{Two $\R^d$-valued random variables $X$ and $Z$ with finite second moment are coupled monotonically if $\E[X \cdot Z] = \MCov(\Law(X),\Law(Z))$.} then $\Vcal(Z) = \Functional(\Law(Z))$, else $\Vcal(Z) \ge \Functional(\Law(Z))$.
        \item \label{it:lem.Fprops.convex} $\Vcal$ is convex and weakly lower semicontinuous.
        \item \label{it:lem.Fprops.minimizer}If $\Functional$ admits an optimizer in $\Pcal_2(\R^d)$, then for every $m\in\mathbb R^d$ there is a unique $Z^\star \in L^2(\Gcal;\R^d)$ with $\mathbb E[Z^\star]=m$,  $\Vcal(Z^\star)= \min_{\alpha \in \Pcal_2(\R^d)}\Functional( \alpha)$, and $(X,Z^\star)$ monotone.
    \end{enumerate}
\end{lemma}

\begin{proof}
    \ref{it:lem.Fprops.Frechet}: 
    Let $(Z_t)_{t \in [0,1]}$ be a continuous curve in $L^2(\mathcal G; \mathbb R^d)$.
    Since $\Law(Z_t + \Gamma) \in \Pcal_2(\R^d)$ is absolutely continuous w.r.t.\ the Lebesgue measure, there is by Brenier's theorem a convex function $v_t \colon \R^d \to \R$ with $\nabla v_t(Z_t + \Gamma) \sim \nu$.
    In particular, the coupling
    \[
        \pi^t := \Law(Z_t + \Gamma, \nabla v_t(Z_t + \Gamma) )
    \]
    is optimal for $\MCov(\beta^t, \nu)$ where $\beta^t := \Law(Z_t + \Gamma) = \Law(Z_t) \ast \gamma_1$.
    Therefore, we deduce from stability of optimal transport maps (see, for example, \cite[Theorem 5.20]{Vi09}) and uniqueness of the optimizer to the optimization problem $\MCov(\beta^t,\nu)$, that $\pi^t \to \pi^0$ in distribution for $t \to 0$.
    At the same time, it is straightforward to see that $\beta^t \to \beta^0$ in total variation.
    Therefore all requirements of \cite[Lemma A.2]{JoPa24} are met and we obtain that $\nabla v_t \to \nabla v_0$ in $\beta^0$-probability.
    We want to use the above to show that $\nabla v_t(Z_t + \Gamma) \to \nabla v_0(Z_0 + \Gamma)$ in probability. 
    To this end, fix $\epsilon > 0$ and note that
    \begin{multline*}
        \mathbb P\left( |\nabla v_t(Z_t + \Gamma) - \nabla v_0(Z_0 + \Gamma)| \ge 3 \epsilon \right)
        \\
        \le
        \underbrace{\mathbb P\left( |\nabla v_t(Z_t + \Gamma) - \nabla v_0(Z_t + \Gamma)| \ge \epsilon \right)}_{=:A(t)}
        +
        \underbrace{\mathbb P\left( |\nabla v_0(Z_t + \Gamma) - \nabla v_0(Z_0 + \Gamma)| \ge 2\epsilon \right)}_{=:B(t)}.
    \end{multline*}
    To deal with the first term of the right-hand side, observe that
    \[
        A(t)
        \le
        \mathbb P\left( |\nabla v_t(Z_0 + \Gamma) - \nabla v_0(Z_0 + \Gamma)| \ge \epsilon \right) + {\rm TV}(\beta^t,\beta^0),
    \]
    where both terms of the right-hand side vanish be our previous arguments.
    In order to handle the second term, we use that $\hat \pi^t := \Law(Z_t + \Gamma, \nabla v_0(Z_t + \Gamma))$ is an optimal coupling with marginals $\beta^t$ and $\nu^t := (\nabla v_0)_\# \beta^t$.
    Since $\beta^t \to \beta^0$ in total variation, the same holds true for the convergence of $\nu^t \to \nu^0 = \nu$.
    As before, using stability of optimal transport maps, we find that $\hat \pi^t \to \hat \pi^0 = \pi^0$ in distribution.
    Hence, we can assume w.l.o.g.\ that there is $\hat Z_t \sim \beta^0$ with 
    \begin{equation}
        \label{eq:first_term}
        |Z_t + \Gamma - \hat Z_t| + |\nabla v_0(Z_t + \Gamma) - \nabla v_0(\hat Z_t))| \to 0
    \end{equation} 
    in probability.
    Consequently,
    \[
        B(t) \le 
        \mathbb P\left( | \nabla v_0(Z_t + \Gamma) - \nabla v_0(\hat Z_t)| \ge \epsilon \right)
        +
        \mathbb P\left( | \nabla v_0(\hat Z_t) - \nabla v_0(Z_0 + \Gamma)| \ge \epsilon \right).
    \]
    Here, letting $t \to 0$, the first term vanishes on the right-hand side vanishes by \eqref{eq:first_term} while the second term vanishes due to \cite[Lemma A.1]{JoPa24}.
    We have shown that $A(t) + B(t) \to 0$ as $t \to 0$, which yields the claim that $\nabla v_t(Z_t + \Gamma) \to \nabla v_0(Z_0 + \Gamma)$ in probability.
    Since the family $(\|\nabla v_t(Z_t + \Gamma)\|_{L^2}^2)_{t \in [0,1]}$ is uniformly integrable as $\nabla v_t(Z_t + \Gamma) \sim \nu$, we have shown that
    \begin{equation}
        \label{eq:L2 convergence}
        \nabla v_t(Z_t + \Gamma) \to \nabla v_0(Z_0 + \Gamma)\quad\text{in }L^2(\mathcal H;\mathbb R^d).
    \end{equation}
    Next, let $Z_1 \in L^2(\Gcal;\R^d)$ and consider in the particular curve $(Z_t)_{t \in [0,1]}$ with 
    \[
        Z_t := (1-t) Z + t Z_1 \quad \text{for } t \in [0,1].
    \]
    As $\pi^0$ maximizes $\MCov(\beta^0, \nu)$ we get the upper bound
    \begin{align*}
        \Vcal(Z_t) - \Vcal(Z_0) &= \E\left[ (Z_t + \Gamma) \cdot \nabla v_t(Z_t + \Gamma) - (Z_0 + \Gamma) \cdot \nabla v_0(Z_0 + \Gamma) - (Z_t - Z_0) \cdot X\right]
        \\
        &\le \E\left[ (Z_t + \Gamma) \cdot \nabla v_t(Z_t + \Gamma) -  (Z_0 + \Gamma)  \cdot \nabla v_t(Z_t + \Gamma)  - (Z_t - Z_0) \cdot X \right] \\
        &= t \E\left[ (Z_1 - Z_0) \cdot \nabla v_t(Z_t + \Gamma) - (Z_1 - Z_0) \cdot X  \right] \\
        &= t \E\left[ (Z_1 - Z_0) \cdot \left( (\nabla v_t \ast \gamma_1)(Z_t)  - X \right)\right],
    \end{align*}
    where the right-hand side  used independence of $Z_t$ and $\Gamma$ for the last equality.
    Similarly, we obtain the lower bound
    \begin{align} \nonumber
        \Vcal(Z_t) - \Vcal(Z_0) &\ge t \E\left[ (Z_1 - Z_0) \cdot \left( \nabla v_0(Z_0 + \Gamma) - X \right) \right] \\
        &= t \E\left[ (Z_1 - Z_0) \cdot\left( (\nabla v_0\ast \gamma_1)(Z_0) - X \right) \right]. \label{eq:Fprops.shown}
    \end{align}
    Combining these two inequalities, dividing by $t$ and then sending $t$ to zero yields
    \[
        \lim_{t \to 0} \frac1t \left( \Vcal(Z_t) - \Vcal(Z_0) \right) = \E\left[ (Z_1 - Z_0) \cdot \left( (\nabla v_0 \ast \gamma_1)(Z_0) - X \right) \right],
    \]
    where we invoked \eqref{eq:L2 convergence} for the upper bound.
    We have shown that $\Vcal$ is Gateaux differentiable with derivative $D_Z \Vcal = \nabla v_0 \ast \gamma_1(Z) - X \in L^2(\mathcal G;\mathbb R^d)$.
    Furthermore, since as consequence of \eqref{eq:L2 convergence} the map $Z \mapsto D_Z \Vcal : L^2(\mathcal G;\mathbb R^d) \to L^2(\mathcal G;\mathbb R^d)$ is continuous, we conclude that $\Vcal$ is Frechet differentiable.

    \ref{it:lem.Fprops.monotone}: If $\Law(X,Z)$ is monotone, then $\E[X\cdot Z] = \MCov(\mu,\alpha)$ where $\alpha:=\Law(Z)$. Thus,
    \begin{align*}
        \Vcal(Z) &= \sup_{\substack{Y \in L^2(\Hcal;\R^d), \\ Y \sim \nu }} \E\left[ (Z + \Gamma) \cdot Y \right] - \MCov(\mu,\alpha) 
        \\
        &= \MCov(\nu,\alpha\ast\gamma_1) - \MCov(\mu,\alpha) = \Functional(\alpha),
    \end{align*}
    proving the claim.

    \ref{it:lem.Fprops.convex}: We have seen in \eqref{eq:Fprops.shown} that $\Vcal(Z_1) \ge \Vcal(Z_0) + D_{Z_0}\Vcal(Z_1 - Z_0)$ for every $Z_0,Z_1 \in L^2(\Gcal;\R^d)$. It readily follows that $\Vcal$ is convex.
    Thus, $\Vcal : L^2(\mathcal G;\mathbb R^d) \to \mathbb R$ is continuous and convex, which implies by \cite[Theorem 9.1]{BaCo11} that $\Vcal$ is also lower semicontinuous.

    \ref{it:lem.Fprops.minimizer}:
    If $\Functional$ admits an optimizer in $\Pcal_2(\R^d)$, then by \cite[Theorem 1.5]{BaSchTsch23} there is a Bass martingale from $\mu$ to $\nu$ with Bass measure $\alpha^\star \in \Pcal_2(\R^d)$, where $\alpha^\star$ is a minimizer of $\Functional$. 
    By \cite[Theorem 1.4]{BBST23}, $\alpha^\star$ is the push-forward of $\mu$ via a Brenier map $\nabla u$. 
    Defining $Z^\star =\nabla u (X)$ we see that $\Vcal (Z^\star)=\Functional(\alpha^\star)$, and by construction $(X,Z^\star)$ is monotone. Since $\Vcal$ is translation invariant, we can set the mean of $Z^\star$ to be any real number, without affecting its optimality or the monotonicity of $(X,Z^\star)$.
\end{proof}

\begin{remark}
    It is somewhat remarkable that despite $\Functional$ being not $\Wcal_2$-differentiable ($\MCov(\cdot,\mu)$ may fail to be differentiable), by passing on to the Hilbert space setting, we obtain a Frechet differentiable functional $\Vcal$.
    The reason for this is that while there may be multiple optimal couplings from $\mu$ to $\alpha := \Law(Z)$ for $Z \in L^2(\Gcal;\R^d)$, by selecting the pair $(X,Z)$ we have already fixed (in the Hilbert space setting) the coupling and thereby avoid this issue.
\end{remark}

\begin{lemma}
    Let $(Z_t)_{t \ge 0}$ be an absolutely continuous curve in $L^2(\Gcal;\R^d)$ with derivative $\frac{d}{dt} Z_t = - D_{Z_t} \Vcal$. 
    Then $t\mapsto\mathbb E[Z_t]$ is a constant.
\end{lemma}

\begin{proof}
    We have $D_{Z_t} \Vcal = \nabla v_t \ast \gamma_1(Z_t) - X $ where $\nabla v_t$ is the Brenier map from $\Law(Z_t + \Gamma)$ to $\nu$. Hence 
    \[ 
        \E[\nabla v_t \ast \gamma_1(Z_t) ] 
        = \mathbb E[ \mathbb E[\nabla v_t (Z_t+\Gamma) | \mathcal G] ]
        = \mathbb E[ \nabla v_t (Z_t+\Gamma)] = \bary(\nu).
    \]
    Since also $\bary(\nu)= \bary(\mu)=\mathbb E[X]$, we conclude that $\frac{d}{dt}\E[Z_t] = 0$ for almost every $t$.
\end{proof}

\section{The Gradient flow}

\subsection{Existence and strong convergence} \label{sec:gradient flow}

To show $L^2$-convergence of the gradient flow, we will impose further assumptions on the marginals $(\mu,\nu)$. 
Throughout this section we denote by 
\[C_\rho:=\text{co(supp $\rho$)},\] 
the convex hull of the support of a measure $\rho \in \Pcal(\R^d)$. Since the set $C_\nu$ features prominently in the subsequent parts of our work, we will simply write \[ C:=C_\nu. \]
Without loss of generality we assume that $dim(C)=d$, otherwise restricting to the affine space spanned by $C$ to achieve this. We also write  \[I:=\text{int}(C),\] 
for the interior of $C$, which is thus a nonempty open subset of $\R^d$. If $\mu\le_{cx} \nu$ then  $C_\mu\subset C$ and $\bary(\nu)=\bary(\mu)$. Hence we may assume w.l.o.g.\ that $\bary(\nu)=\bary(\mu)=0 \in I$ and write $\mathcal P_2^0(\mathbb R^2)$ for the subset of $\mathcal P_2(\mathbb R^2)$ consisting of centred measures.

Subsequently we will work with the following assumptions: \medskip

\noindent \textbf{Assumption (A):} We have
$\mu\le_{cx} \nu$ and
  \begin{enumerate}[label = (A\arabic*)]
    \item \label{it:A1} the pair $(\mu,\nu)$  is irreducible;
    \item \label{it:A2} $\nu$ is compactly supported, i.e.\ $C$ is compact;
    \item \label{it:A3} $\supp(\mu) \subseteq I$, i.e. $C_\mu$ is compactly contained in $I$. 
\end{enumerate}  

To stress the importance of this assumption, we recall the following fact. 
If we only assume \ref{it:A1} and \ref{it:A2}, i.e., the pair $(\mu,\nu)$ is irreducible and $\nu$ is in addition compactly supported, the Bass functional may fail to admit a minimizer in $\Pcal_2(\R^d)$ and minimizing sequences $(\alpha_n)_{n \in \N}$ may even fail to be tight (see \cite[Example 6.7]{BBST23}
).
On the other hand, assuming \ref{it:A3} in addition to \ref{it:A1} and \ref{it:A2}, this drawback cannot happen anymore, as shown in the subsequent lemma.
For reasons of presentation we present these results already here while the proofs rely on some technicalities, which are shown in the subsequent section, and are therefore deferred to Subsection \ref{ssec:proofs}.

\begin{lemma}
    \label{lem:existence_minimizer}
    Under Assumption $(A)$ the functional $\Vcal$ admits a minimizer in $L^2$ and $\Functional$ a minimizer in $\Pcal_2(\R^d)$. Furthermore, any minimizer of $\Vcal$ is bounded.
\end{lemma}

As we have existence of a minimizer to the Bass functional under Assumption (A), the classical theory of gradient flows in Hilbert spaces provides us with the next result.

\begin{lemma}
\label{lem:Brezis_Bruck}
    Let $Z_0 \in L^2(\mathcal G;\R^d)$. Under Assumption $(A)$
     there exists a curve $(Z_t)_{t \ge 0}$ in $L^2(\mathcal G;\R^d)$ starting in $Z_0$ such that 
    \begin{align} \label{eq:Brezis_Bruck}
        \frac{d}{dt} Z_t = - D_{Z_t} \Vcal \text{ for a.e.\ }t > 0, \quad \text{and} \quad \frac{d^+}{dt} Z_t = - D_{Z_t} \Vcal \text{ for }t > 0.
    \end{align}
    Further, there exists a minimizer $Z^\star \in L^2(\Gcal;\R^d)$ of $\Vcal$ such that $\lim_{t \to \infty} Z_t = Z^\star$ weakly in $L^2$ and $\lim_{t \to \infty} \Vcal(Z_t) = \Vcal(Z^\star)$.
\end{lemma}

We have already seen in Lemma \ref{lem:Brezis_Bruck} that, under Assumption (A), for any initial datum $Z_0 \in L^2(\Gcal;\R^d)$ the $L^2$-gradient flow $(Z_t)_{t \ge 0}$ of $\Vcal$ exists and converges weakly to a minimizer of $\Vcal$, which exists by Lemma \ref{lem:existence_minimizer}.
But even beyond that, Assumption (A) is sufficiently strong to ensure strong convergence of the gradient flow.
We now state the main result of this section.

\begin{theorem}\label{thm:convergence}
    Under Assumption (A), let $(Z_t)_{t \ge 0}$ be an $L^2$-gradient flow of $\Vcal$.
    Then $(Z_t)_{t \ge 0}$ converges strongly to the unique minimizer $Z^\star$ of $\Vcal$ with $\mathbb E[Z^\star] = \mathbb E[Z_0]$.
\end{theorem}

For the purpose of Proposition \ref{prop:boundedness}, which is key in showing Theorem \ref{thm:convergence}, we need to pass to a pathwise absolutely continuous version of the gradient flow:

Fix a representative $\tilde Z_0$ of $Z_0$ and, for fixed $\omega$, consider the curve $\tilde Z = (\tilde Z_t)_{t \ge 0}$ with
\begin{equation}
    \label{eq:modification}
    \tilde Z_t := \tilde Z_0 - \int_0^t D_{Z_s} \Vcal \, ds,
\end{equation}
which is $\mathbb P$-almost surely well-defined since $|D_{Z_s} \Vcal|$ is uniformly bounded.
Consequently, $\tilde Z_t$ is a representative of $Z_t$ with absolutely continuous paths having pathwise the derivative
\begin{equation}
    \label{eq:modification.derivative}
    \frac{d\tilde Z_t}{dt} = - D_{Z_t} \Vcal = - D_{\tilde Z_t} \Vcal. 
\end{equation}

\begin{proposition}\label{prop:boundedness}
    Under Assumption (A) let $(Z_t)_{t \ge 0}$ be an $L^2$-gradient flow of $\Vcal$ and $(\tilde Z_t)_{t \ge 0}$ the representative given in \eqref{eq:modification}.
    Then there is $M > 0$ such that the random time $\tau := \inf \{ t > 0 : |\tilde Z_t| \le M \}$ is a.s.\ finite and satisfies almost surely
    \begin{enumerate}[label = (\roman*)]
        \item $t \mapsto |\tilde Z_t|$ is strictly decreasing for $t \in [0,\tau)$;
        \item $|\tilde Z_t| \le M$ for all $t \ge \tau$.
    \end{enumerate}
    In particular, $(|Z_t|^2)_{t \ge 0}$ is uniformly integrable.
\end{proposition}



The proofs of Theorem \ref{thm:convergence} and Proposition \ref{prop:boundedness} need some preparations and will be given at the end of this section.

\subsection{Geometric auxiliary results}

As it turns out, under Assumption (A) and assuming that the starting variable $Z_0$ is in $L^\infty(\Gcal;\R^d)$, one can show that the $L^2$-Bass gradient flow is confined to a bounded subset of $\R^d$, which eventually allows us to establish its convergence in the norm of $L^2$.
Therefore, we need to control the gradient of $\Vcal$, which we do by analysing properties of the Brenier map $\nabla v : \R^d \to \R^d$ that pushes a measure $\beta \in \Pcal_2(\R^d)$ with $\beta \ll \lambda$ to $\nu$.
To this end, we consider the support function $\ell : \R^d \to \R$ of $C$ given by
\begin{equation}
    \label{eq:ell}
    \ell(a) := \max \{ y \cdot a : y \in C \}.
\end{equation}
We define the slice $S_{a,\delta}$ of $C$ consisting of points in $C$ that are at most $\delta > 0$ away from the supporting hyperplane with normal direction $a$, that is,
\begin{equation}
    \label{eq:segment}
    S_{a,\delta} := \{ y \in C : y \cdot a > \ell(a) - \delta \}.
\end{equation}
Clearly, $\ell$ is Lipschitz continuous on the $(d-1)$-sphere $S^{d-1}$. 
For the next lemma, we recall that we have assumed (w.l.o.g.) that $C$ spans $\R^d$.
\begin{lemma} \label{lem:S}
    Let $\delta > 0$ and $r \ge 0$. Under Assumption (A), we have
    \begin{enumerate}[label = (\roman*)]
        \item \label{it:S.1} there is $\epsilon = \epsilon(\delta,C) > 0$ such that for all $a,b \in S^{d-1}$
        \begin{equation}
            \label{eq:S.1}
            |a - b| \le \epsilon \implies S_{a,\delta / 2} \subseteq S_{b,\delta};
        \end{equation}
        \item \label{it:S.2} for $a \in S^{d-1}$ the set $\{ y \in C : \exists y' \in S_{a,\delta} \text{ with }|y - y'| \le r \}$ is contained in $S_{a,\delta + r}$;
        \item \label{it:S.3} the $\nu$-measure of each $\delta$-slice is uniformly bounded away from zero, i.e., 
        \[ \eta(\delta) := \inf \{ \nu(S_{a,\delta}) : a \in S^{d-1} \} > 0; \]
        \item \label{it:S.4} we have 
        \begin{equation}
            \label{eq:Delta}
            \Delta := \inf \{ \ell(a) - x \cdot a :  a \in S^{d-1}, x \in C_\mu \} > 0,
        \end{equation}
        so that, for $0 < \delta < \Delta$ the sets $C_\mu$ and $S_{a,\delta}$ are disjoint, for all $a \in S^{d-1}$.
    \end{enumerate}
\end{lemma}

\begin{proof}
Ad \ref{it:S.1}, let $a,b \in S^{d-1}$ and $\epsilon := \frac{\delta}{4L}$ where $L := \max \{ |y| : y \in C \}$. We claim that then \eqref{eq:S.1} holds.
Indeed, if $|a - b| \le \epsilon$ and $y \in S_{a,\delta/2}$ we have
\begin{align*}
    y \cdot b \ge y \cdot a - L | b - a | > \ell(a) - \frac{\delta}{2} - L | b - a | \ge \ell(b) - \frac{\delta}{2} - 2 L | b - a | \ge \ell(b) - \delta.
\end{align*}
Ad \ref{it:S.2}, let $a \in S^{d-1}$, $y \in C$ and $y' \in S_{a,\delta}$ with $|y - y'| \le r$. We merely observe that
\begin{align*}
    y \cdot a \ge y' \cdot a - r > \ell(y') - \delta - r.
\end{align*}
Ad \ref{it:S.3}, first observe that $\nu(S_{a,\delta}) > 0$ for all $a \in S^{d-1}$, since $C$ is the smallest closed convex set containing the support of $\nu$ and $C \setminus S_{a,\delta}$ is a closed convex subset of $C$.
We proceed by contradiction, that means, we assume by compactness that there is a sequence $(a_n)_{n \in \N}$ in $S^{d-1}$ converging to $a$ with $\nu(S_{a_n,\delta}) \to 0$.
By \ref{it:S.1} we have that whenever $|a_n - a| \le \epsilon$ we have $S_{a,\delta / 2} \subseteq S_{a_n,\delta}$ which means that
\[
    \liminf_{n \to \infty} \nu(S_{a_n,\delta}) \ge \nu(S_{a,\delta / 2}) > 0,
\]
yielding a contradiction, hence, $\eta(\delta) > 0$.

Ad \ref{it:S.4}, assume that $\Delta = 0$ which means by compactness that there is $a \in S^{d-1}$ and $x \in C_\mu$ with $\ell(a) - x \cdot a = 0$. Consequently, $x$ is a support point of $C$ with normal vector $a$ and thus, by \cite[Corollary 7.6]{BaCo11}, $x$ is in the boundary of $C$ which contradicts that $x \in C_\mu \subseteq I$.
 \end{proof}

Before continuing we would like to recall that as $\nu$ is compactly supported under Assumption (A), hence the Brenier map from any absolutely continuous $\beta \in \Pcal_1(\R^d)$ to $\nu$ exists.
Thus we can work in this section with probability measures having only finite first moments.

\begin{lemma}\label{lem:Brenier boundary}
    Let $\Lambda \subseteq \mathcal P_1(\R^d)$ be a tight family consisting of measures equivalent to the Lebesgue measure and write $\nabla v_\beta$ for the Brenier map from $\beta \in \Lambda$ to $\nu$.
    Under Assumption (A), for $\delta > 0$, there is a constant $M = M(\nu,\Lambda,\delta) > 0$ such that
    \[
        z \in \R^d, |z| \ge M \implies \forall \beta \in \Lambda, \nabla v_\beta(z) \in S_{\frac{z}{|z|}, \delta}.
    \]
\end{lemma}

\begin{proof}
    Using Lemma \ref{lem:S} \ref{it:S.3} and tightness of $\Lambda$, there exists a radius $r > 0$ such that
    \begin{equation}
        \label{eq:beta_ineq}
        \inf_{\beta \in \Lambda} \beta(B(r)) > 1 - \eta(\delta / 2),
    \end{equation}
    where $B(r) := \{ y \in \R^d : |y| \le r \}$ is the centred ball with radius $r$.
    
    Fix $\beta \in \Lambda$ and $z \in \R^d$ with $|z| \ge \sup \{ |y| : y \in C \} =: L$, and write $a := \frac{z}{|z|} \in S^{d-1}$.
    We proceed to show that if $x := \nabla v_\beta(z) \notin S_{a,\delta}$ then $|z| \le M := 2 (r + L)^2 / \delta$.
    It follows from $(\nabla v_\beta)_\# \beta = \nu$ that
    \[
        \beta(\{ \tilde z \in \R^d : \nabla v_\beta(\tilde z) \notin S_{a,\delta/2} \} ) = 1 - \nu(S_{a,\delta/2}) \le 1 - \eta(\delta/2),
    \]
    whence, by \eqref{eq:beta_ineq} there exists $z' \in B(r)$ with $x' := \nabla v_\beta(z') \in S_{a,\delta/2}$.
    We have
    \begin{align} \label{eq:Brenier boundary.1}
        |z' - x| &\le r + L, \\ 
        \label{eq:Brenier boundary.2}
        |z - x| &\ge a \cdot (z - x) = |z| - x \cdot a \ge |z| - \ell(a) + \delta > 0,
    \end{align}
    where we used in \eqref{eq:Brenier boundary.2} that $|z| \ge L \ge \ell(a)$ and $x \notin S_{a,\delta}$.
    By cyclical monotonicity we must have
    \begin{align*}
        |x - z|^2 + |x' - z'|^2 \le |x' - z|^2 + |x - z'|^2,
    \end{align*}
    implying that
    \begin{align*}
        (|z| - \ell(a) + \delta)^2 &\le
        |z|^2 - 2 z \cdot (x') + |x'|^2 + (L+r)^2 \\
        &= |z|^2 - 2 |z| a \cdot x' + |x'|^2 + (L+r)^2 \\
        &\le |z|^2 - 2 |z| (\ell(a) - \delta/2) + |x'|^2 + (L+r)^2,
    \end{align*}
    where the first equality comes from \eqref{eq:Brenier boundary.1} and \eqref{eq:Brenier boundary.2} and the last is due to $x' \in S_{a,\delta/2}$.
    Rearranging the terms leads to
    \[
        |z| \delta + \delta^2 \le L^2 + (L+r)^2,
    \]
    from where we readily derive the claim.
\end{proof}

\begin{lemma}\label{lem:Brenier boundary_2}
    Under the assumptions of Lemma \ref{lem:Brenier boundary}, there is a (possibly larger) constant $M = M(\nu,\Lambda,\delta) > 0$ such that
    \begin{equation}
        \label{eq:Brenier boundary_2}
        z \in \R^d, |z| \ge M \implies \forall \beta \in \Lambda, \nabla v_\beta \ast \gamma_1(z) \in S_{\frac{z}{|z|},\delta}.
    \end{equation}
\end{lemma}

\begin{proof}
    By Lemma \ref{lem:Brenier boundary} there is $\tilde M > 0$ such that
    \begin{equation}
        \label{eq:Brenier boundary_2.1}
        z \in \R^d, |z| \ge \tilde M \implies \forall \beta \in \Lambda, \nabla v_\beta(z) \in S_{\frac{z}{|z|}, \delta / 4}.
    \end{equation}
    Further, by Lemma \ref{lem:S} \ref{it:S.1} there is $\epsilon > 0$ such that for all $a,b \in S^{d-1}$
    \begin{equation}
        \label{eq:Brenier boundary_2.2}
        |a-b| \le \epsilon \implies S_{a, \delta/4} \subseteq S_{b,\delta/2}.
    \end{equation}
    Next, we choose $r$ sufficiently large such that $\gamma_1(B(r)^c) \le \delta / (4L)$ and set 
    \[ M := \max\left(\tilde M, \frac{2r}{\epsilon} + r\right). \]
    It remains to show that $M$ has the property described in \eqref{eq:Brenier boundary_2}.
    Fix $z \in \R^d$ with $|z| \ge M$ and $\beta \in \Lambda$, and write $x := \nabla v_\beta \ast \gamma_1(z)$.
    We have
    \[
        x = \gamma_1(B(r)) \underbrace{\int_{ B(r) } \nabla v_\beta(z - y) \frac{\gamma_1(dy)}{\gamma_1(B(r))}}_{=: x_1} + \gamma_1(B(r)^c) \underbrace{\int_{ y \in B(r)^c } \nabla v_\beta(z - y) \frac{\gamma_1(dy}{\gamma_1(B(r)^c)}}_{=: x_2}.
    \]
    When $|y| \le r$ then $b = \frac{z}{|z|}$ and $a = \frac{z - y}{|z - y|}$ satisfy
    \begin{align*}
        |a - b| \le \left| \frac{z}{|z|} - \frac{z}{|z - y|} \right| + \frac{r}{|z - y|} \le \frac{| |z - y| - |z| |}{|z - y|} + \frac{r}{M - r} \le \frac{2r}{M - r} \le \epsilon.
    \end{align*}
    Therefore, \eqref{eq:Brenier boundary_2.1} and \eqref{eq:Brenier boundary_2.2} yield that in this case $\nabla v_\beta(z - y) \in S_{b,\delta / 2}$, hence, also $x_1 \in S_{b,\delta / 2}$ as the latter is a convex set.
    On the other hand, we have
    \[
        |x_2| \gamma_1(B(r)^c) \le L \gamma_1(B(r)^c) \le \frac{\delta}{4}.
    \]
    We claim that $x \in S_{b,\delta}$. Indeed,
    \begin{align*}
        x \cdot b &\ge \gamma_1(B(r)) (x_1 \cdot b) - \frac{\delta}{4} 
        \ge \left( \gamma_1(B(r)) x_1 + \gamma_1(B(r)^c) x_2 \right) \cdot b \\
        &> \left(1 - \frac{\delta}{4L} \right) \left(\ell(b) - \frac\delta2 \right) - \frac{\delta}{4}
        \ge \ell(b) - \frac{3\delta}{4} - \frac{\delta \ell(b)}{4L} \ge \ell(b) - \delta,
    \end{align*}
    where we used that $\ell(b) \le L$ and $x_1 \in S_{b,\delta / 2}$.
    This concludes the proof.
\end{proof}

\begin{lemma} \label{lem:Brenier boundary mu}
    Under the assumptions of Lemma \ref{lem:Brenier boundary}, let $M > 0$ satisfy \eqref{eq:Brenier boundary_2}. Then we have
    \begin{equation}
        \label{eq:Brenier boundary mu}
        x \in C_\mu, z \in \R^d, |z| \ge M \implies \forall \beta \in \Lambda, z \cdot (\nabla v_\beta \ast \gamma_1(z) - x) \ge |z| ( \Delta - \delta),
    \end{equation}
    where $\Delta$ is given as in \eqref{eq:Delta}.
\end{lemma}

\begin{proof}
    Let $x \in C_\mu$, $\beta \in \Lambda$ and $z \in R^d$ with $|z| \ge M$, and write $a = \frac{z}{|z|}$.
    \begin{align*}
        z \cdot \left( \nabla v_\beta \ast \gamma_1(z) - x \right) 
        &= |z| a \cdot \left( \nabla v_\beta \ast \gamma_1(z) - x \right) 
        \\
        &\ge
        |z| \left( a \cdot \nabla v_\beta \ast \gamma_1(z) - \ell(a) + \Delta \right) \\
        &\ge |z| (\Delta - \delta),
    \end{align*}
    where we use Lemma \ref{lem:S} \ref{it:S.4} for the first inequality and Lemma \ref{lem:Brenier boundary_2} for the second.
\end{proof}

\subsection{Postponed proofs} \label{ssec:proofs}

\begin{proof}[Proof of Lemma \ref{lem:existence_minimizer}]
    We have established in Item \ref{it:lem.Fprops.monotone} of Lemma \ref{lem:Vcal} that $\Vcal(Z) \ge \Functional(\Law(Z))$ for all $Z \in L^2(\mathcal G;\mathbb R^d)$ with equality if and only if the law of $(X,Z)$ is monotone.
    Hence, $Z$ minimizes $\Vcal$ if and only if $\Law(X,Z)$ is monotone and $\Law(Z)$ is a Bass measure.

    Since $\nu$ is compactly supported and $(\mu,\nu)$ is irreducible, there exists by \cite[Theorem 1.4]{BBST23} a convex function $v : \mathbb R^d \to \mathbb R$, $u := v \ast \gamma_1$ such that $\alpha := (\nabla u^\ast)_\# \mu$ is the Bass measure of the Bass martingale joining $\mu$ and $\nu$.
    Using Lemma \ref{lem:S} \ref{it:S.4}, we pick $\delta > 0$ such that $\delta < \Delta$. Then we have $C_\mu \cap S_{a,\delta} = \emptyset$ for all $ a \in S^{d-1}$.
    By Lemma \ref{lem:Brenier boundary_2} there is $M > 0$ such that
    \[ \{ z \in \R^d : |z| \ge M \} \subseteq \bigcup_{a \in S^{d-1}} \{ z \in \R^d : \nabla u(z) \in S_{a,\delta} \}. \]
    We find that $\{ z \in \R^d : \nabla u(z) \in C_\mu \} \subseteq \{ z \in \R^d : |z| < M \}$, from where we conclude that $\alpha$ is supported on $\{ z \in \R^d : |z| \le M \}$ and $\alpha \in \Pcal_2(\R^d)$.
    Hence, $Z = \nabla u^\ast(X) \in L^2(\mathcal G;\R^d)$ minimizes $\Vcal$.
\end{proof}

\begin{proof}[Proof of Lemma \ref{lem:Brezis_Bruck}]
    The existence of the gradient flow satisfying \eqref{eq:Brezis_Bruck} follows by \cite[Theorems 3.1]{Bre72}. 
    The weak convergence to $Z^\star$, a minimizer of $\Vcal$, is a consequence of \cite[Theorem 4]{Bru75}. 
    To check that $\Vcal(Z_t)\to \Vcal(Z^\star)$, we will use that $\|\frac{d^+}{dt}Z_t\|_{L^2} \to 0$, where the latter property is due to \cite[Theorem 3.7]{Bre72}, and argue by contradiction. Indeed, if $(t_n)_{n \in \N}$ is an increasing, non-negative sequence with $t_n\to\infty$ and $a>0$ exist such that $\Vcal(Z_{t_n})\geq a +\Vcal(Z^\star)$, then by convexity
    \[
        -a\geq \Vcal(Z^\star)-\Vcal(Z_{t_n}) \geq D_{Z_{t_n}} \Vcal \cdot( Z^\star - Z_{t_n})  =- \frac{d^+}{dt} Z_{t_n} \cdot( Z^\star - Z_{t_n})\to 0,
    \]
    for $n \to \infty$, as $(Z_t)_{t \ge 0}$ is $L^2$-bounded. This contradicts $a>0$.
\end{proof}

\begin{proof}[Proof of Proposition \ref{prop:boundedness}]
    Due to Lemma \ref{lem:S} \ref{it:S.4} we can pick $\delta > 0$ such that $\delta < \Delta$. We know by Lemma \ref{lem:Brezis_Bruck} that the gradient flow is weakly convergent to a minimizer of $\Vcal$.
    Therefore, $\sup_{t \ge 0} \mathbb E[ |Z_t|^2 ] < \infty$ and $\Lambda := \{ \Law(Z_t) \ast \gamma_1 : t \ge 0 \}$ is tight.
    Let $M > 0$ be the constant provided by Lemma \ref{lem:Brenier boundary mu}, then we have by \eqref{eq:Brenier boundary mu} that
    \[
        |\tilde Z_t| \ge M \implies \tilde Z_t \cdot (\nabla v_t \ast \gamma_1(\tilde Z_t) - X) \ge |\tilde Z_t| (\Delta - \delta),
    \]
    where $\nabla v_t : \R^d \to \R^d$ denotes the Brenier map from $\Law(Z_t) \ast \gamma_1$ to $\nu$.
    Recall that by Lemma \ref{lem:Vcal} we have $D_{\tilde Z_t} \Vcal = \nabla v_t \ast \gamma_1(\tilde Z_t) - X$.
    By \eqref{eq:modification.derivative} we have on a $\mathbb P$-full set $\tilde \Omega$, for almost every $t$
    \[
        |\tilde Z_t| \ge M \implies \frac12 \frac{d}{dt} | \tilde Z_t |^2 = - \tilde Z_t \cdot \left( \nabla v_t \ast \gamma_1(\tilde Z_t) - X \right) \leq - |\tilde Z_t| (\Delta - \delta) <  0,
    \]
    where we used for the last inequality that $\Delta - \delta > 0$.
    Hence, on $\{ t \ge 0 : t < \tau \}$ the curve $t \mapsto |\tilde Z_t|$ is strictly decreasing on $\tilde \Omega$. By the same argument, $\tau$ must be finite.
    On the other hand, let $\tau_1$ be a random time with $\tau_1 \ge \tau$ and $|\tilde Z_{\tau_1}| > M$ on $\{ \tau_1 > \tau \}$, then we have on $\{\tau_1 > \tau\} \cap \tilde \Omega$
    \[
        0 < |\tilde Z_{\tau_1}|^2 - |\tilde Z_\tau|^2 < -(\tau_1 - \tau) (\Delta - \delta)M,
    \]
    which yields that $\{\tau_1 > \tau\}$ is disjoint with $\tilde \Omega$.
    Consequently, we have shown that $|\tilde Z_t| \le M$ for all $t \ge \tau$ on $\tilde \Omega$, which completes the proof.
\end{proof}

\begin{proof}[Proof of Theorem \ref{thm:convergence}]
    By Proposition \ref{prop:boundedness} we have that $(|Z_t|^2)_{t \ge 0}$ is uniformly integrable.
    In turn, the family $\{ \Law(Z_t) : t \ge 0 \}$ is $\Wcal_2$-precompact which enables us to find an increasing, non-negative sequence $(t_n)_{n \in \N}$ with $t_n \to \infty$ such that $\Law(Z_{t_n}) \to \alpha$ in $\Wcal_2$.
    Using Lemma \ref{lem:Vcal} \ref{it:lem.Fprops.monotone} and Lemma \ref{lem:Brezis_Bruck} we have
    \[
        \Functional(\Law(Z^\star)) = \Vcal(Z^\star) = \lim_{n \to \infty} \Vcal (Z_{t_n}) \ge \lim_{n \to \infty} \Functional(\Law(Z_{t_n})) = \Functional(\alpha).
    \]
    Since $Z^\star$ is the unique minimizer of $\Vcal$ with mean $\mathbb E[Z_0]$ this shows that $\Law(Z^\star) = \alpha$.
    It follows immediately that
    \[
        \lim_{t \to \infty} \mathbb E[|Z_t|^2] = \mathbb E[|Z^\star|^2],
    \]
    which, together with weak convergence of $(Z_t)_{t \ge 0}$, yields $L^2$-norm convergence to $Z^\star$.
\end{proof}

\section{The one-dimensional case}

Throughout this section, we assume that $d = 1$.

\subsection{Second-order analysis}

\label{sec:one-dim-Hess}

In this part we assume that $d=1$, although we will still often use vector and matrix notation to make this part consistent with the other ones.
Throughout this section, $\Phi_1$ denotes the standard Gaussian cumulative distribution function and $\phi_1:=\Phi_1'$ its density.

\begin{lemma}
    \label{lem:conditional}
    Let $\Delta Z \in L^1(\Gcal;\R^d)$, $Z$ be $\Gcal$-measurable. Then we have for $\gamma_1$-a.e.\ $\zeta$
    \begin{equation}
        \label{eq:conditional}
        \E[\Delta Z | Z + \Gamma = \zeta ] = \frac{\E[\phi_1(\zeta - Z) \Delta Z]}{\E[\phi_1(\zeta - Z)]}.
    \end{equation}
\end{lemma}

\begin{proof}
    Using independence of $\mathcal G$ and $\Gamma$, we compute for measurable, bounded $h \colon \R \to \R$ that
    \begin{align*}
        \E[\Delta Z h(Z + \Gamma)] 
        &= \E\left[ \Delta Z \E[ h(Z + \Gamma) | \mathcal G] \right]
        = \E\left[ \Delta Z \int h(\zeta) \phi_1(\zeta - Z) \, d\zeta \right]
        \\
        &= \int \left( \frac{\E[\Delta Z \phi_1(\zeta - Z)]}{\E[\phi_1(\zeta - Z)]} \right) h(\zeta) \E[\phi_1(\zeta - Z)] \, d\zeta
        \\
        &= \E \left[ \int \left( \frac{ \E[\Delta Z \phi_1(\zeta - Z)]}{\E[\phi_1(\zeta - Z)]} \right) h(\zeta) \phi_1(\zeta - Z) \, d\zeta \right] \\
        &= \E \left[ \left( \frac{ \E[\Delta Z \phi_1(\zeta - Z)]}{\E[\phi_1(\zeta - Z)]} \right)\Big|_{\zeta = Z + \Gamma} h(Z + \Gamma) \right].
    \end{align*}
    From this equality and uniqueness of conditional expectations, we readily derive \eqref{eq:conditional}.
\end{proof}

To proceed with our second-order analysis of $\Vcal$, we require additional regularity of $\nu$. \medskip

\noindent {\bf Assumptions (A'):} The pair $(\mu,\nu)$ satisfies Assumption (A), and in addition
\begin{enumerate}[label = (A\arabic*), start = 4]
    \item \label{it:A4} $\nu$ is absolutely continuous with $\lambda\text{-}\operatorname{ess\;inf}_{x \in I} \frac{d\nu}{d\lambda}(x) > 0$;
    \item \label{it:A5} the density of $\nu$ satisfies $\lambda\text{-}\operatorname{ess\;sup}_{x \in \R} \frac{d\nu}{d\lambda}(x) < \infty$.
\end{enumerate}

Using this assumption, we can compute the derivative of $D_Z \Vcal$ along continuously differentiable curves. Throughout we denote by $Hv$ the second derivative (Hessian) of $v$. 

\begin{proposition} \label{prop:derivative}
    Under Assumption (A'), let $(Z_t)_{t \ge 0}$ be a continuously differentiable curve in $L^2(\Gcal;\R^d)$ with derivative $G_t$.
    Then $t \mapsto D_{Z_t} \Vcal$ is continuously differentiable with
    \begin{equation}
        \label{eq:derivative}
        \frac{d}{dt} \left( D_{Z_t} \Vcal \right) = \E \left[ Hv_t(Z_t + \Gamma) \left( G_t - \E[G_t | Z_t + \Gamma]  \right) | \Gcal \right],
    \end{equation}
    where $\nabla v_t : \R \to \R$ is the Brenier map from $\Law(Z_t + \Gamma)$ to $\nu$.
\end{proposition}

\begin{proof}
    As we are currently working with $d = 1$, the Brenier map is explicitly given by $\nabla v_t = Q_\nu \circ F_t$ where $Q_\nu$ is the quantile function of $\nu$ and $F_t$ is the cdf of $Z_t$.
    It follows from \ref{it:A4} that $Q_\nu$ is almost surely differentiable on $(0,1)$.
    Consequently, $\frac{d}{dx} \nabla v_t(x)$ and $\frac{d}{dt} \nabla v_t(x)$ exist for $\lambda$-a.e.\ $x \in \R$.
    We compute the space and time derivatives of $\nabla v_t$, and obtain $\lambda$-a.s.\
    \begin{align}\label{eq:1dim Hessian}
        Hv_t = \frac{d}{dx} (Q_\nu \circ F_t) = \Big(\frac{d}{dx} Q_\nu\Big) \circ F_t \; \Big(\frac{d}{dx}F_t\Big) = \Big(\frac{d}{dx} Q_\nu\Big) \circ F_t \; \E[\phi_1(\cdot - Z_t)],   
    \end{align}
    as well as ($\gamma_1$-a.s.)
    \begin{align*}
        \frac{d}{dt} (\nabla v_t) &= \Big(\frac{d}{dx} Q_\nu\Big) \circ F_t \; \Big(\frac{d}{dt} F_t\Big) = - \Big(\frac{d}{dx} Q_\nu\Big) \circ F_t \;  \E[\phi_1(\cdot- Z_t)G_t] \\
        &= - Hv_t \; \E[G_t | Z_t + \Gamma = \cdot],
    \end{align*}
    where the last equality follows by Lemma \ref{lem:conditional}.
    As $\frac{d\nu}{dx}$ is $\lambda$-essentially bounded away from zero on $I$ by some constant $c > 0$, we have that
    \[
        \lambda\text{-}\esssup_{x \in \R} |Hv_t(x)| \le \frac1c < \infty \quad\text{and}\quad
        \lambda\text{-}\esssup_{x \in \R} 
        \Big|\frac{d}{dt} \nabla v_t(x)\Big| \le \frac1c \| G_t \|_{L^2} < \infty.
    \]
    Recall the form of $D_{Z_t} \Vcal$ proved in Lemma \ref{lem:Vcal}.
    Thanks to the bounds, we can exchange integration with differentiation and get
    \begin{align*}
        \frac{d}{dt} (D_{Z_t} \Vcal) &= \E\left[ \frac{d}{dt} \left( \nabla v_t(Z_t + \Gamma) \right) \Big| \Gcal \right] 
        = \E\left[ \Big( \frac{d}{dt} \nabla v_t\Big)(Z_t + \Gamma) + Hv_t(Z_t+\Gamma) G_t \Big| \Gcal \right]
        \\
        &= \E\left[ Hv_t(Z_t + \Gamma) \big(G_t - \E[G_t | Z_t + \Gamma]\big) \big| \Gcal \right].
    \end{align*}
    Note that the maps $t \mapsto Hv_t(Z_t + \Gamma)$ and $t \mapsto (\frac{d}{dt} \nabla v_t)(Z_t + \Gamma)$ are $L^2$-continuous, from where we conclude that $t \mapsto D_{Z_t} \Vcal$ is continuously differentiable.
\end{proof}

\begin{lemma} \label{lem:second derivative}
    Under Assumption (A'), let $(Z_t)_{t \in [0,1]}$ be a curve in $L^2(\Gcal;\R^d)$ with $Z_t = Z_0 + t \Delta Z$ where $\Delta Z := Z_1 - Z_0$.
    Then we have
    \begin{equation}
        \label{eq:second derivative}
        \frac{d^2}{dt^2} \Vcal(Z_t) = \E\left[ \left( \Delta Z - \E[\Delta Z | Z_t + \Gamma] \right) Hv_t(Z_t + \Gamma) \left( \Delta Z - \E[\Delta Z | Z_t + \Gamma] \right) \right],
    \end{equation}
    where $\nabla v_t$ is the Brenier map from $\Law(Z_t + \Gamma)$ to $\nu$.
\end{lemma}

\begin{proof}
    We have shown in Proposition \ref{prop:derivative} that $t \mapsto D_{Z_t} \Vcal$ is continuously differentiable.
    Therefore, the second derivative of $t \mapsto \Vcal(Z_t)$ can be represented as
    \begin{align*}
        \frac{d^2}{dt^2} \Vcal(Z_t) &=
        \frac{d}{dt} \E\left[\Delta Z D_{Z_t} \Vcal\right] =
        \E\left[ \Delta Z \frac{d}{dt} D_{Z_t} \Vcal\right]
        \\
        &= \E\left[ \Delta Z Hv_t(Z_t + \Gamma) \left(\Delta Z - \E[\Delta Z | Z_t + \Gamma] \right) \right],
    \end{align*}
    where we used \eqref{eq:derivative} and the tower property for the last equality.
    Further, since
    \[
        \E\left[ \E\left[ \Delta Z | Z_t + \Gamma \right] Hv_t(Z_t + \Gamma) \left( \Delta Z - \E\left[\Delta Z | Z_t + \Gamma\right] \right) \right] = 0,
    \]
    we conclude with \eqref{eq:second derivative}.
\end{proof}

\subsection{Contraction property of conditional expectation}

The final ingredient in the proof of exponential convergence of the gradient flow of the lifted Bass functional is the following contraction property of the conditional expectation.

\begin{lemma}\label{lem:reverse_ineq}
    Let $\Delta Z \in L^2(\mathcal G;\R)$ with $\E[\Delta Z] = 0$ and $R > 0$.
    Then there is $\epsilon = \epsilon(R) \in (0,1)$ such that, for every $Z \in L^2(\mathcal G;\R)$ with $\| Z\| \le R$ we have
    \[
        \| \mathbb E[ \Delta Z | Z + \Gamma] \|_{L^2}^2 \le (1 - \epsilon) \|\Delta Z\|_{L^2}^2.
    \]
\end{lemma}

\begin{proof}
    As preparation we introduce the quantities, for $\zeta \in \R$,
    \[
        g(\zeta) := \min_{(y,z) \in \R \times \R, \, |y|, |z| \le R} \frac{\phi_1(\zeta - y)}{\phi_1(\zeta - z)}
    \]
    Note that $g$ is continuous and $g \in (0,1)$. Therefore, we can set
    \[
        \epsilon := \min_{z \in \R, \, |z| \le R} \E[g(z + \Gamma)],
    \]
    and by the previous observations we see that $\epsilon > 0$.
    Furthermore, since $\E[\Delta Z ] = 0$ we get by Lemma \ref{lem:conditional}
    \begin{align*}
        \E[\Delta Z | Z+\Gamma = \zeta] &= \frac{\E[\phi_1(\zeta - Z)\Delta Z]}{\E[\phi_1(\zeta - Z)]}
        = \E\left[ \Delta Z \left( \frac{\phi_1(\zeta - Z)}{\E[\phi_1(\zeta - Z)]} - g(\zeta) \right) \right].
    \end{align*}
    Since $\frac{\phi_1(\zeta - Z)}{\E[\phi_1(\zeta - Z)]} \ge g(\zeta)$ and $g(\zeta) < 1$, we have that
    \[
        h(Z) := \left( \frac{\phi_1(\zeta - Z)}{\E[\phi_1(\zeta-Z)]} - g(\zeta) \right) \frac{1}{1 - g(\zeta)},
    \]
    is a probability density.
    Therefore by Jensen's inequality we get
    \begin{align*}
        \E[\Delta Z | Z + \Gamma = \zeta]^2 &= (1 - g(\zeta))^2 \; \E[\Delta Z h(Z)]^2
        \\
        &\le (1 - g(\zeta))^2 \; \E[\Delta Z^2 h(Z)] \\
        &= (1 - g(\zeta)) \E\left[\Delta Z^2 \left( \frac{\phi_1(\zeta - Z)}{\E[\phi_1(\zeta - Z)]} - g(\zeta) \right)\right]
        \\
        &\le (1 - g(\zeta)) \E\left[ \Delta Z^2 | Z + \Gamma = \zeta \right].
    \end{align*}
    Hence, taking expectations leads to
    \begin{align*}
        \| \E[\Delta Z | Z + \Gamma] \|_{L^2}^2 
        &\le
        \E\left[ (1 - g(Z + \Gamma)) \E[\Delta Z^2 | Z + \Gamma] \right]
        \\
        &= \E\left[ (1 - g(Z+\Gamma)) \Delta Z^2 \right]
        \\
        &= \E\left[ \Delta Z^2 \E[1 - g(Z+\Gamma) | \Gcal] \right]
        \\
        &\le \E\left[ \Delta Z^2 \left(1 - \min_{z \in \R, \, |z| \le R} \E[g(z + \Gamma)] \right) \right] = (1 - \epsilon) \E[\Delta Z^2],
    \end{align*}
    which concludes the proof.
\end{proof}

\subsection{Exponential convergence}

We consider the setting and notation of Section \ref{sec:one-dim-Hess}, wherein $d=1$. As in that part, we use here vector and matrix notation nevertheless.

\begin{lemma}\label{lem:Hess_lower}
    Let $\nu \ll \lambda$ satisfy $\frac{d\nu}{d\lambda} > 0$ on $I$, $R > 0$ and $Z \in L^2(\Gcal;\R)$ with $|Z| \le R$.
    Denote by $\nabla v_Z$ the Brenier map from $\Law(Z + \Gamma)$ to $\nu$.
    Then we have for $\lambda$-a.e.\ $\zeta$
    \begin{equation}
        \label{eq:Hess_lower}
        Hv_Z(x)\geq \frac{\inf_{\zeta \in \R, \, |\zeta|\leq R}\phi_1(x-\zeta)}{\|\frac{d\nu}{d\lambda}\|_{L^\infty(\lambda)}}.
    \end{equation}
\end{lemma}

\begin{proof}
    Write $F$ for the cdf of $\Law(Z + \Gamma)$.
    Under our assumption on $\nu$ the quantile function $Q_\nu$ of $\nu$ is almost everywhere differentiable on $(0,1)$ and we have by \eqref{eq:1dim Hessian} that for $\lambda$-a.e.\ x
    \[
        Hv_Z(x) = \Big(\frac{d}{dx} Q_\nu \Big) \circ F(x) \; \E[ \phi_1(x - Z) ].
    \]
    As consequence of $\frac{d\nu}{d\lambda} > 0$ on $I$ we have $\lambda$-a.s.
    $(\frac{d}{dx} Q_\nu) \circ F \ge \frac{1}{\frac{d\nu}{d\lambda}} \ge \frac1{\|\frac{d\nu}{d\lambda}\|_{L^\infty(\lambda)}}$.
    Further, since $|Z| \le R$ we get $\E[\phi_1(x-Z)] \ge \inf_{\zeta \in \R, \, |\zeta| \le R} \phi_1(x-\zeta)$, we obtain \eqref{eq:Hess_lower}.
\end{proof}

\begin{lemma}\label{lem:Hess_lower_2}
Under Assumption (A') there is for every $R > 0$ an $\epsilon = \epsilon(R) > 0$ such that for all $Z_i \in L^2(\Gcal;\R)$, $i = 0,1$ with $\esssup|Z_i| \le K$, the curve $(Z_t)_{t \in [0,1]}$ with $Z_t:=(1-t)Z_1+tZ_0$ satisfies
\begin{equation}
    \label{eq:Hess_lower_2}
    \frac{d^2}{dt^2}\Vcal(Z_t)\geq \epsilon \|Z_1 - Z_0 - \E[Z_1-Z_0|Z_t+\Gamma]\|_{L^2}^2.
\end{equation}
\end{lemma}
\begin{proof}
   We have shown in Lemma \ref{lem:second derivative} that, after denoting $\Delta Z:=Z_1-Z_0$,
   \[
        \frac{d^2}{dt^2}\Vcal(Z_t) = \E[ Hv_t(Z_t+\Gamma)\left( \Delta Z - \E[\Delta Z | Z_t + \Gamma] \right)^2 ],
   \]
   where $\nabla v_t$ is the Brenier map from $\Law(Z_t + \Gamma)$ to $\nu$. 
   Using the lower bound derived in Lemma \ref{lem:Hess_lower} we have
    \begin{align}\nonumber
        \frac{d^2}{dt^2}\Vcal(Z_t) 
        &\geq \E\left[ g_R(Z_t+\Gamma) \left( \Delta Z - \E[\Delta Z | Z_t + \Gamma] \right)^2 \right]
        \\
        \label{eq:Hess_lower_2.1}
        &= \E\left[ g_R\ast \gamma_1(Z_t) \left( \Delta Z - \E[\Delta Z | Z_t + \Gamma] \right)^2 \right],
    \end{align}
    where for $x \in \R$
    \[ g_R(x):= \frac{\inf_{\zeta \in \R, \, |\zeta|\leq R}\phi_1(x-\zeta)}{\|\frac{d\nu}{d\lambda} \|_{L^\infty(\lambda)}}. \] 
    Observe that $g_R$ is continuous and strictly positive everywhere. 
    Hence, $g_R\ast \gamma_1$ is also continuous and strictly positive everywhere. 
    It follows that for $\epsilon :=\inf_{|\zeta|\leq R}g_R\ast\gamma_1(\zeta)>0$, so \eqref{eq:Hess_lower_2} is readily derived from \eqref{eq:Hess_lower_2.1}.
\end{proof}

\begin{lemma}\label{lem:strong_conv}
    Let $R > 0$.
    Under Assumption (A') there is $\epsilon>0$ such that for all $Z_i\in L^2(\mathcal G;\R)$, $i =0,1$ with $|Z_i| \le R$ we have
    \begin{equation}
        \label{eq:strong convexity}
        \Vcal(Z_1) - \Vcal(Z_0) - \langle D_{Z_0} \Vcal, Z_1 - Z_0 \rangle_{L^2} \ge \epsilon \| Z_1 - Z_0 \|^2_{L^2}.
    \end{equation}
\end{lemma}

\begin{proof}
    Set $\Delta Z := Z_1 - Z_0$ and consider the curve $(Z_t)_{t \in [0,1]}$ with $Z_t := Z_0 + t \Delta Z$.
    Since $\Vcal$ is continuously differentiable by Lemma \ref{lem:Vcal} we have
    \[
        \Vcal(Z_1) - \Vcal(Z_0) = 
        \int_0^1 \langle D_{Z_t} \Vcal, \Delta Z \rangle_{L^2} \, dt.
    \]
    Hence, it suffices to prove the inequality
    \[
        \langle D_{Z_t}\Vcal - D_{Z_0}\Vcal, \Delta Z\rangle_{L^2} \ge 2\epsilon t \| \Delta Z \|^2_{L^2} .
    \]
    By Lemma \ref{lem:Hess_lower_2} there is $\epsilon(R) > 0$ such that
    \[
        \frac{d}{dt} \langle D_{Z_t} \Vcal, \Delta Z \rangle_{L^2} = \frac{d^2}{dt^2} \Vcal(Z_t) \ge \epsilon(R) \| \Delta Z - \E[\Delta Z | Z_t + \Gamma] \|^2_{L^2}.
    \]
    Finally, by Lemma \ref{lem:reverse_ineq} there is $\epsilon' > 0$ such that
    \[
        \| \Delta Z - \E[\Delta Z | Z_t + \Gamma] \|^2_{L^2} = \| \Delta Z \|^2_{L^2} - \| \E[\Delta Z | Z_t + \Gamma] \|^2_{L^2} \ge \epsilon' \| \Delta Z \|^2_{L^2},
    \]
    where the first equality is due to the tower property.
    Therefore, \eqref{eq:strong convexity} holds for $2\epsilon := \epsilon(R) \epsilon'$, which completes the proof.
\end{proof}

\begin{proof}[Proof of Theorem \ref{thm:second_main} ]
    Per Proposition \ref{prop:boundedness} the whole flow $(Z_t)_{t \ge 0}$ takes values in $B_\infty(M) := \{ Z \in L^2(\Gcal;\R) : \esssup |Z| \le M \}$ for some $M > 0$.
    Denote by $Z^\ast$ the $L^2$-limit of $(Z_t)_{t \ge 0}$ which is a minimizer of $\Vcal$.
    Hence, throughout this proof we may assume that all random variables are uniformly bounded by $M$.
    By Lemma \ref{lem:strong_conv}, we can find $\epsilon$ such that $\Vcal$ is $(2\epsilon)$-strongly convex restricted to $B_\infty(M)$. 
    Using \eqref{eq:strong convexity} we obtain for $Z \in B_\infty(M)$ that
    \[
        \Vcal(Z^\ast) - \Vcal(Z) \ge \inf_{Z'} \langle D_Z\Vcal, Z' - Z \rangle_{L^2} + \epsilon \|Z' - Z \|^2_{L^2},
    \]
    from where we derive
    \begin{equation}\label{eq:aux_PL}
        \Vcal(Z)-\Vcal(Z^\star) \leq \frac{1}{4\epsilon}\|D_Z\Vcal\|_{L^2}^2.
    \end{equation}
    On the other hand, adding \eqref{eq:strong convexity} (where we let $(Z_0,Z_1) = (Z,Z^\star)$) to \eqref{eq:strong convexity} (where we let $(Z_1,Z_0) = (Z,Z^\star)$) yields
    \begin{equation}
        \label{eq:aux_norm_2}
        \langle D_Z\Vcal, Z- Z^\star \rangle_{L^2} \geq 2\epsilon \| Z-Z^\star \|_{L^2}^2.
    \end{equation}
    From \eqref{eq:aux_PL} and continuous differentiability of $(Z_t)_{t \ge 0}$ with $\frac{d}{dt} Z_t = - D_{Z_t} \Vcal$ we get
    \[
        \frac{d}{dt}\big( \Vcal(Z_t)-\Vcal(Z^\star) \big) 
        = 
        -\|D_{Z_t}\Vcal\|_{L^2}^2 \leq -4\epsilon \big( \Vcal(Z_t)-\Vcal(Z^\star) \big),
    \]
    while \eqref{eq:aux_norm_2} yields
    \[\frac{d}{dt} \|Z_t - Z^*\|_{L^2}^2= - 2\langle Z_t - Z^*, D_{Z_t} \Vcal  \rangle_{L^2} \leq - 4\epsilon \| Z_t- Z^*\|_{L^2}^2  . \]
    We conclude by Gronwall's inequality.
\end{proof}

\bibliographystyle{abbrv}
\bibliography{joint_biblio-2}

\end{document}